\documentclass[12pt,a4paper,oneside]{amsart}
\usepackage[left=1.1in, right=0.8in, top=1.0in, bottom=1.0in]{geometry}
\usepackage{graphicx}
\usepackage{setspace}
\usepackage{indentfirst} 
\usepackage{cases}
\usepackage{wrapfig}
\usepackage[toc,page]{appendix}
\usepackage{amsmath}
\usepackage{amsthm}
\usepackage{amssymb}
\usepackage{enumerate}
\usepackage{mathrsfs}
\usepackage{dcolumn}
\usepackage{xcolor}
\usepackage{bookmark}
 \usepackage{url}
\usepackage{enumitem}
\usepackage{bookmark}
\usepackage[nocompress]{cite}
\usepackage{hyperref}
\usepackage[capitalize,nameinlink]{cleveref}

\newcolumntype{L}{D{.}{.}{2,5}}
\theoremstyle{plain}
\newtheorem{thm}{Theorem}
\newtheorem*{mainthm}{Hadamard's Lemma}

\newtheorem{remark}[thm]{Remark}

\newtheorem{corollary}[thm]{Corollary}

\DeclareMathOperator{\spn}{span}

\theoremstyle{definition}

\begin{document}
	\title[Hadamard's lemma in separable Hilbert spaces]{Hadamard's lemma in separable Hilbert spaces}
	\author{{A. B\"erd\"ellima}}
\thanks{E-mail: berdellima@gmail.com}
\AtEndDocument{\bigskip{%
	  \textsc{German International University in Berlin\\\phantom f Faculty of Engineering\\ \phantom f Berlin 13507, Germany.}}}

	\begin{abstract} We extend Hadamard's Lemma to the setting of a separable Hilbert space.
	\end{abstract}
	
		\maketitle

\section{Introduction}
\label{s1}
\begin{mainthm}\cite[\S2, pp.17]{Nestruev}
\label{th:Hadamard-lemma}
Let $\mathbb E^n$ be an $n$-dimensional Euclidean space and let $f$ be a smooth, real-valued function defined on an open, star-convex neighborhood $U$  of a point $a\in \mathbb E^n$. Then $f(x)$ can be expressed, for all $x\in U$, in the form: 
\begin{equation}
\label{eq:Hadamard-lemma}
f(x)=f(a)+\sum_{k=1}^n(x_k-a_k)\,g_k(x),
\end{equation}
where each $g_k$ is a smooth function on $U,\,a=(a_1,a_2,\ldots,a_n)$, and $x=(x_1,x_2,\ldots,x_n)$. 
\end{mainthm}
In it's simpliest form, when $\mathbb E=\mathbb R$ and $a=0$, this lemma reduces to the expression $$f(x)=f(0)+x\,g(x).$$
Applying Hadamard's Lemma iteratively $n$ times yields an expression of the $n$-th order Taylor series expansion of $f$ around the origin, where the remainder is the product of $x^{n+1}$ with another smooth function $g$
$$f(x)=f(0)+x\,f'(0)+x^2\,\frac{f''(0)}{2!}+x^3\,\frac{f'''(0)}{3!}+\ldots+x^n\,\frac{f^{(n)}(0)}{n!}+x^{n+1}\,g(x).$$ 
As simple as this lemma might appear, it has deep implications in smooth infinitessimal analysis, as explored in \cite[\S1, \S2]{Moerdijk--Reyes}. 
For a concrete illustration, consider the ring of dual numbers $\mathbb R[\varepsilon]$, which consists of elements of the form $a+\varepsilon\,b$, where $a,b\in\mathbb R$ and $\varepsilon\neq 0$ is an infinitesimal element satisfying $\varepsilon^2=0$. 
Let $C^{\infty}(\mathbb R)$ denote the space of smooth functions on $\mathbb R$.
For $f\in C^{\infty}(\mathbb R)$, by Hadamard's Lemma the first order Taylor expansion around $a\in\mathbb R$ is given by $$f(x)=f(a)+(x-a)\,f'(a)+(x-a)^2\,g(x),$$ where $g\in C^{\infty}(\mathbb R)$.
Thus, it becomes possible to evaluate a smooth function $f$ at a dual number $a+\varepsilon\,b$, yielding:
$$f(a+\varepsilon\,b)=f(a)+\varepsilon\,bf'(a)+\varepsilon^2\,b^2g(a+\varepsilon\,b)=f(a)+\varepsilon\,bf'(a).$$

Hadamard's Lemma has found several applications in various areas of mathematics.
For instance, in \cite{Perov}, the lemma was used to derive a representation for nonlinear Lipschitz continuous differential operators. In \cite{Osawa}, it was employed to obtain variational formulas for the ground state value of the semi-linear Poisson equation.

In this note, we aim to extend Hadamard's Lemma to functions defined on an infinite-dimensional vector space. As a primary example, we consider the case of a separable Hilbert space. Several results analogous to those in finite-dimensional Euclidean spaces will be derived as a consequence of this extension.

\section{Preliminaries}
By $\mathcal H$, we denote a Hilbert space equipped with its inner product $\langle\cdot,\cdot\rangle$, which induces the canonical norm $\|\cdot\|=\langle \cdot,\cdot\rangle^{1/2}$. We denote the dual of $\mathcal H$ by $\mathcal H^*$.
A function $f:\mathcal H\to  \mathbb R$ is Fr\'echet differentiable at $x\in\mathcal H$, if there exists a bounded linear functional $T(x):\mathcal H\to\mathcal H^*$ such that 
\begin{align*}
\lim_{\|h\|\to 0}\frac{|f(x+h)-f(x)-\langle T(x),h\rangle|}{\|h\|}=0.
\end{align*}
If $f$ is  Fr\'echet differentiable at every point in an open set $U\subseteq\mathcal H$, we say that $f$ is Fr\'echet differentiable on $U$. When $f$ is Fr\'echet differentiable at $x$, we call the linear operator $T$ the gradient of $f$ at $x$, denoted $\nabla f(x)$. 
A function $f:\mathcal H\to  \mathbb R$ is Gateaux differentiable at $x\in\mathcal H$, if there exists a continuous linear function $g:\mathcal H\to\mathbb R$ such that 
\begin{align*}
g(v)=\lim_{\delta\to 0}\frac{f(x+\delta\,v)-f(x)}{\delta},\quad \text{for all}\;v\in\mathcal H.
\end{align*}
If $f$ is Fr\'echet differentiable at $x$, then it is also Gateaux differentiable at $x$ with $g(v)=\langle \nabla f(x), v\rangle$ for all $v\in\mathcal H$. 
For $n\geq 2$, we define the higher order derivatives of $f$ recursively.
A function $f:\mathcal H\to\mathbb R$ is $n$-times differentiable at $x$, if it is $(n-1)$-times differentiable and there exists a multilinear functional $T(x):\mathcal H\to (\mathcal H^*)^n$ such that 
\begin{align*}
\lim_{\|h_n\|\to 0}\frac{|\nabla^{n-1}f(x+h_n)(h_1,\ldots, h_{n-1})-\nabla^{n-1}f(x)(h_1,\ldots, h_{n-1})-T(x)(h_1,\ldots, h_n)|}{\|h_n\|}=0
\end{align*}
uniformly for $h_1,h_2,\ldots, h_{n-1}$ in bounded sets in $\mathcal H$. We denote $T=\nabla^nf$. A function $f:\mathcal H\to\mathbb R$ is $n$-times differentiable on an open set $U\subseteq\mathcal H$, if it is $n$-times Fr\'echet differentiable at every $x\in U$. A function $f:\mathcal H\to\mathbb R$ is smooth on $U$ if its derivatives of all orders exists at every $x\in U$. We denote the set of all smooth functions by $C^{\infty}(\mathcal H)$. For more on infinite dimensional analysis, we refer to \cite[\S12]{Bogachev--Smolyanov}.

Let $\mathcal H$ be a separable Hilbert space and let $(u_k)$ be an orthonormal basis in $\mathcal H$. A mapping $g:\mathcal H\to\mathcal H$ is smooth on an open set $U\subseteq \mathcal H$, if $g_k(x)=\langle g(x), u_k\rangle$ is a smooth function on $U$ for every $k\in\mathbb N$. Finally a set $U\subseteq \mathcal H$ is a star-convex neighoborhood of $a\in\mathcal H$ if for all $x\in U$ and $t\in[0,1]$ it holds that $(1-t)a+tx\in U$. A Hilbert space $\mathcal H$ is real, if its associated scalar field is $\mathbb R$.
To this end, we assume that the Hilbert space $\mathcal H$ is real and separable.

\section{Hadamard's lemma in separable Hilbert spaces}

\begin{thm}
\label{th:Hadamard-lemma-2}
Let $f$ be a smooth, real-valued function defined on an open, star-convex neighborhood $U$  of a point $a\in \mathcal H$. Then $f(x)$ can be expressed, for all $x\in U$, in the form: 
\begin{equation}
\label{eq:Hadamard-lemma-2}
f(x)=f(a)+\langle g(x), x-a\rangle,
\end{equation}
where $g$ is a smooth mapping on $U$. 
\end{thm}
 \begin{proof}
 Let $h(t)=f(a+t(x-a)$ for $t\in[0,1]$. Because $f$ is smooth on $U$, in particular it is Gateaux differentiable on $U$. Because $U$ is an open star-convex neighborhood of $a$, then for $s$ sufficiently small $a+(t+s)(x-a)\in U$. Smoothness of $f$ on $U$ implies that the limit 
 \begin{align*}
 \lim_{s\to 0}\frac{f(a+(t+s)(x-a))-f(a+t(x-a))}{s}
 \end{align*}
 exists and it equals
 \begin{align*}
 \langle \nabla f(a+t(x-a)),x-a\rangle,
 \end{align*}
 where $\nabla f$ is the gradient of $f$. On the other hand by definition of $h(t)$ we have that
 \begin{align*}
 h'(t)=\lim_{s\to 0}\frac{h(t+s)-h(t)}{s}=\lim_{s\to 0}\frac{f(a+(t+s)(x-a))-f(a+t(x-a))}{s},
 \end{align*}
 therefore
 \begin{align*}
 h'(t)=\langle \nabla f(a+t(x-a)),x-a\rangle.
 \end{align*}
 Let $(u_k)$ be an orthonormal basis for $\mathcal H$. Then 
 \begin{align*}
 h'(t)=\sum_{k=1}^{\infty}\nabla_k f(a+t(x-a))\,(x_k-a_k),
 \end{align*}
 where $\nabla_kf(a+t(x-a))=\langle \nabla f(a+t(x-a)), u_k\rangle$ and $x_k-a_k=\langle x-a,u_k\rangle$ for every $k\in\mathbb N$. Integrating over the interval $[0,1]$ and noting that $h(0)=f(a)$ and $h(1)=f(x)$ yields
 \begin{align*}
 f(x)-f(a)=\int_0^1\sum_{k=1}^{\infty}\nabla_k f(a+t(x-a))\,(x_k-a_k)\,dt.
 \end{align*}
 The integral above is absolutely convergent. Indeed note that
 \begin{align*}
 \Big|\sum_{k=1}^{\infty}\nabla_k f(a+t(x-a))\,(x_k-a_k)\Big|&\leq \Big(\sum_{k=1}^{\infty}(\nabla_k f(a+t(x-a)))^2\Big)^{1/2}\,\Big(\sum_{k=1}^{\infty}(x_k-a_k)^2\Big)^{1/2}\\&
 =\|\nabla f(a+t(x-a))\|\,\|x-a\|
 \end{align*}
 implies 
 \begin{align*}
 \int_0^1 \Big|\sum_{k=1}^{\infty}\nabla_k f(a+t(x-a))\,(x_k-a_k)\Big|\,dt&\leq \|x-a\|\int_0^1\|\nabla f(a+t(x-a))\|\,dt\\& \leq \|x-a\|\max_{0\leq t\leq 1}\|\nabla f(a+t(x-a))\|,
\end{align*}  
where we have used the fact that $(\nabla f)\circ \psi: [0,1]\times\mathcal H\times\mathcal H\to\mathcal H$, with $\psi(t,a,x)=a+t(x-a)$, is a continuous mapping in all its three arguments, and in particular in the variable $t$ over the compact interval $[0,1]$. Interchanging summation with integration implies
 \begin{align}
 \label{eq:expansion-f}
 f(x)-f(a)=\sum_{k=1}^{\infty}(x_k-a_k)\,\int_0^1\nabla_k f(a+t(x-a))\,dt=\sum_{k=1}^{\infty}(x_k-a_k)\,g_k(x),
 \end{align}
 where 
 \begin{align*}
 g_k(x)=\int_0^1\nabla_k f(a+t(x-a))\,dt,\quad\text{for all}\;k\in\mathbb N.
 \end{align*}
 Because $f$ is smooth on $U$, then so is $g_k$ for every $k\in\mathbb N$. Consequently the mapping 
 \begin{align*}
 g(x)=\sum_{k=1}^{\infty}g_k(x)\,u_k
 \end{align*}
 is smooth on $U$. Rearranging terms in \eqref{eq:expansion-f} yields \eqref{eq:Hadamard-lemma-2}.
 \end{proof}

 \section{Consequences of Hadamard's Lemma}
 \subsection{Taylor series expansion}
 \begin{thm}
 \label{th:Taylor-expansion}
 Any smooth function
$f$ in a star-convex neighborhood $U$ of a point $a\in\mathcal H$ is representable in the
form
\begin{align}
\label{eq:Taylor-expansion}
f(x)&=f(a)+\langle \nabla f(a), x-a\rangle + \frac{1}{2!}\,\nabla^2f(a)(x-a,x-a)+\ldots\\\nonumber&\ldots+\frac{1}{n!}\nabla^nf(a)(\underbrace{x-a,x-a,\ldots, x-a}_{n-\text{times}})+g(x)(\underbrace{x-a,x-a,\ldots, x-a}_{(n+1)-\text{times}}),
\end{align}
where $g:\mathcal H\to\mathcal H^{n+1}$ is a smooth mapping and satisfies the representation
\begin{align}
\label{eq:g-expansion}
g(x)(x-a,x-a,\ldots,x-a)=\sum_{\substack{i\in\mathbb N^{n+1}\\ |\alpha_i|=n+1}}(x_i-a_i)^{\alpha_i}g_{\alpha_i}(x).
\end{align}
Here $i=(i_1,i_2,\dots,i_{n+1}),\,\alpha_i=(\alpha_{i_1}, \alpha_{i_2},\ldots, \alpha_{i_{n+1}})$ are multi-indices, and 
$(x_i-a_i)^{\alpha}=(x_{i_1}-a_{i_1})^{\alpha_{i_1}}\,(x_{i_2}-a_{i_2})^{\alpha_{i_2}}\ldots(x_{i_{n+1}}-a_{i_{n+1}})^{\alpha_{i_{n+1}}}.$
By smoothness of the mapping $g$ we mean that $g_{\alpha_i}$ are smooth functions on $U$ for every multi-index $\alpha_i$ satisfying $|\alpha_i|=n+1$.
 \end{thm}
 
 \begin{proof}
 From Theorem \ref{th:Hadamard-lemma-2} we have 
$$f(x)=f(a)+\sum_{k=1}^{\infty}(x_k-a_k)\,g_k(x),$$
where $g_k$ is smooth on $U$ for every $k\in\mathbb N$. Applying Theorem \ref{th:Hadamard-lemma-2} on each $g_k$ yields
\begin{align*}
f(x)&=f(a)+\sum_{k=1}^{\infty}(x_k-a_k)\,\Big(g_k(a)+\sum_{j=1}^{\infty}(x_j-a_j)\,g_{kj}(x)\Big)
\end{align*}
Realizing that $g_k(a)=\nabla_kf(a)$ for every $k\in\mathbb N$, rearranging terms gives
\begin{align*}
f(x)
=f(a)+\langle \nabla f(a), x-a\rangle+\sum_{k,j}(x_k-a_k)(x_j-a_j)\,g_{kj}(x).
\end{align*}
Now applying Theorem \ref{th:Hadamard-lemma-2} on each $g_{kj}$ yields
\begin{align*}
f(x)
=f(a)+\langle \nabla f(a), x-a\rangle+\sum_{k,j}(x_k-a_k)(x_j-a_j)\,\Big(g_{kj}(a)+\sum_{i=1}^{\infty}(x_i-a_i)\,g_{kji}(x)\Big).
\end{align*}
Realizing that $2\,g_{kj}(a)=\nabla^2_{kj}f(a)$ and rearranging terms yield
\begin{align*}
f(x)
&=f(a)+\langle \nabla f(a), x-a\rangle+\frac{1}{2}\sum_{k,j}(x_k-a_k)(x_j-a_j)\nabla^2_{kj}f(a)\\&+\sum_{k,j,i}(x_k-a_k)(x_j-a_j)(x_i-a_i)\,g_{kji}(x)\\&
=f(a)+\langle \nabla f(a), x-a\rangle+\frac{1}{2}\nabla^2f(a)(x-a,x-a)\\&+\sum_{k,j,i}(x_k-a_k)(x_j-a_j)(x_i-a_i)\,g_{kji}(x).
\end{align*}
Applying iteratively Theorem \ref{th:Taylor-expansion} on each $g_{kji}$ and so on, yields formula \eqref{eq:Taylor-expansion}.
 \end{proof}
\subsection{Certain representations of a smooth function} 
 
 \begin{thm}
 \label{c:vanishin-on-subspaces}Let $(u_n)$ be an orthonormal basis in $\mathcal H$ and
let $f:\mathcal H\to\mathbb R$ be a smooth function on $\mathcal H$. If $\nabla^jf$ vanishes on every subspace $S_j=\spn\{u_{n_1},u_{n_2},\ldots,u_{n_j}\}$, where $1\leq j\leq k$ for some $k\in\mathbb N$, then $f$ has the following form 
\begin{align}
\label{eq:f-shape-subspaces}
f(x)=\sum_{\substack{i\in\mathbb N^{k+1}\\ |\alpha_i|=k+1}}x_i^{\alpha_i}g_{\alpha_i}(x),\quad \text{for some smooth functions}\;g_{\alpha_i},
\end{align}
where $\alpha_i=(\alpha_{i_1}, \alpha_{i_2},\ldots, \alpha_{i_{k+1}})$ and $i=(i_1,i_2,\dots,i_{k+1})$. 
 \end{thm}
 
 \begin{proof}
 By Theorem \ref{th:Taylor-expansion} we have with $a=0$ the $k$-order Taylor expansion
 \begin{align*}
f(x)&=f(0)+\langle \nabla f(0), x\rangle + \frac{1}{2!}\,\nabla^2f(0)(x,x)+\ldots\\\nonumber&\ldots+\frac{1}{n!}\nabla^nf(0)(\underbrace{x,x,\ldots, x}_{k-\text{times}})+g(x)(\underbrace{x,x,\ldots, x}_{(k+1)-\text{times}}),
\end{align*}
 where $g:\mathcal H\to\mathcal H^{k+1}$ is a smooth mapping and satisfies the representation
\begin{align}
\label{eq:g-expansion}
g(x)(x-a,x-a,\ldots,x-a)=\sum_{\substack{i\in\mathbb N^{k+1}\\ |\alpha_i|=k+1}}(x_i-a_i)^{\alpha_i}g_{\alpha_i}(x),
\end{align}
where $i=(i_1,i_2,\dots,i_{k+1}),\,\alpha_i=(\alpha_{i_1}, \alpha_{i_2},\ldots, \alpha_{i_{k+1}})$. It suffices to show that $\nabla^jf(0)=0$ for all $1\leq j\leq k$. We show by induction on $j\leq k$. Let $j=1$. 
Because $f$ vanishes on every subspace $S_1=\spn\{u_n\}$, then it vanishes in particular at $x=0$. Because $f$ is smooth, then its Gateaux derivative at $x=0$ along each $u_n$ exists and it is given by $\langle\nabla f(0),u_n\rangle$. By definition of Gateaux derivative we have that 
$$\langle\nabla f(0),u_n\rangle=\lim_{\delta\to 0}\frac{f(0+\delta\,u_n)-f(0)}{\delta}.$$
Because $f\equiv 0$ on $\spn\{u_n\}$, then $f(\delta\,u_n)=0$ for every $\delta\in\mathbb R$, consequently the last limit is identically zero. Hence $\langle\nabla f(0),u_n\rangle=0$ for every $n\in\mathbb N$, implying $\nabla f(0)=0$. Let $\nabla^jf=0$ on every $S_j=\spn\{u_{n_1},u_{n_2},\ldots,u_{n_j}\}$ for $1<j<k$. Consider $j+1\leq k$. By definition of the Fr\'echet differentiability we have
\begin{align*}
\lim_{\|h_{j+1}\|\to 0}\frac{|\nabla^{j}f(0+h_{j+1})(h_1,\ldots, h_{j})-\nabla^{j}f(0)(h_1,\ldots, h_{j})-\nabla^{j+1}f(0)(h_1,\ldots, h_{j+1})|}{\|h_{j+1}\|}=0
\end{align*}
uniformly for $h_1,h_2,\ldots, h_{j}$ in bounded sets in $\mathcal H$. Let $v\in\mathcal H$ and write $h_{j+1}=\delta\,v$. Then as $\|h_{j+1}\|\to 0$, it follows that $\delta\to 0$. Rewrite the last limit as
\begin{align*}
\lim_{\delta\to 0}\frac{|\nabla^{j}f(0+\delta\,v)(h_1,\ldots, h_{j})-\nabla^{j}f(0)(h_1,\ldots, h_{j})-\nabla^{j+1}f(0)(h_1,\ldots, \delta\,v)|}{\delta}=0.
\end{align*}
Denote the function $f_j(x)=\nabla^{j}f(x)(h_1,\ldots, h_{j})$ and express $$\nabla^{j+1}f(0)(h_1,\ldots, h_j, x)=\langle \nabla^{j+1}f(0)(h_1,\ldots, h_{j}), x\rangle$$ for $x\in\mathcal H$. The limit is then equivalent to
\begin{align*}
\lim_{\delta\to 0}\frac{|f_j(0+\delta\,v)-f_j(0)-\langle \nabla^{j+1}f(0)(h_1,\ldots, h_{j}), \delta\,v\rangle|}{\delta}=0
\end{align*}
for all $v\in\mathcal H$. Setting $v=u_n$ and rearranging terms yields
\begin{align*}
\lim_{\delta\to 0}\frac{f_j(0+\delta\,u_n)-f_j(0)}{\delta}=\langle \nabla^{j+1}f(0)(h_1,\ldots, h_{j}), u_n\rangle.
\end{align*}
By induction hypothesis $\nabla^jf(0)=0$. On the other hand the inclusion $\spn\{u_n\}\subseteq\spn\{u_n, u_{n_1}, u_{n_2},\ldots, u_{n_{j-1}}\}$ holds, implying $\nabla^{j}f(0+\delta\,u_n)=0$. Consequently $f_j(0+\delta\,u_n)=f_j(0)=0$, and so
\begin{align*}
\langle \nabla^{j+1}f(0)(h_1,\ldots, h_{j}), u_n\rangle=0,\quad\text{for all}\,n\in\mathbb N.
\end{align*}
This reduces to the case $j=1$, hence $\nabla^{j+1}f(0)(h_1,\ldots, h_{j})=0$ for all $j$-tuples $(h_1,\ldots, h_{j})$ in bounded sets in $\mathcal H$ and in particular $\nabla^{j+1}f(0)(u_{n_1},u_{n_2},\ldots,u_{n_j})=0$ for all $j$-tuples of orthonormal elements $(u_{n_1},u_{n_2},\ldots,u_{n_j})$, therefore $\nabla^{j+1}f(0)=0$. 
\end{proof}

\begin{corollary}
\label{c:vanishin-on-axis}
Let $f:\mathcal H\to\mathbb R$ be a smooth function such that $f$ vanishes on $S_n=\spn\{u_n\}$ for all $n\in\mathbb N$. Then $f$ has the following form 
\begin{align}
\label{eq:f-shape}
f(x)=\sum_{k,j}x_kx_j\,g_{kj}(x),\quad \text{for some smooth functions}\;g_{kj}.
\end{align} 
\end{corollary}

 \begin{thm}
 \label{th:sum-two-fnct}
 Any smooth function $f$ such that $f(y)=f(z)=0$, where $y\neq z$, satisfies
 \begin{equation}
 \label{eq:sum-two-fnct}
 f(x)=\sum_{k=1}^{\infty}g_k(x)\,h_k(x),
 \end{equation}
 where $g_k(y)=h_k(z)=0$ for every $k\in\mathbb N$.
 \end{thm}
 \begin{proof}Let $(u_k)$ be an orthonormal base in $\mathcal H$. 
Up to an invertible affine transformation we may assume that $y=u_n$ for some $n\in\mathbb N$ and $z=0$. By Theorem \ref{th:Hadamard-lemma-2} any smooth function $f$ with $f(z)=0$ is representable in the form 
 $$f(x)=\langle \widetilde g(x), x\rangle$$
 for some smooth mapping $\widetilde g:\mathcal H\to\mathcal H$. Let $w\in\mathcal H$ such that $\widetilde g(y)=w$. We may write $\widetilde g(x)=(\widetilde g(x)-w)+w$, then 
 \begin{align*}
 \langle \widetilde g(x), x\rangle=\langle\widetilde g(x)-w, x\rangle+\langle w, x\rangle.
 \end{align*}
 The first term vanishes when $x=z$ and again when $x=y$. Therefore we only need to look at the second term. When $x=y$ we get $0=f(y)=\langle w, y\rangle=w_n$. For $k\neq n$ we write $x_k=x_kx_n-x_k(x_n-1)$. Define $g_{3k-2}(x)=\widetilde g_k(x)-w_k, g_{3k-1}(x)=w_kx_kx_n$, and $g_{3k}(x)=x_n-1$ for all $k\in\mathbb N$. Similarly let $h_{3k-2}(x)=x_k, h_{3k-1}(x)=x_n$, and $h_{3k}(x)=w_kx_k$ for every $k\in\mathbb N$. Then we can express $f$ as follows
 \begin{align*}
 f(x)&=\langle\widetilde g(x)-w, x\rangle+\langle w, x\rangle\\&=\sum_{k=1}^{\infty}g_{3k-2}(x)h_{3k-2}(x)+\sum_{k=1}^{\infty}g_{3k-1}(x)h_{3k-1}(x)+\sum_{k=1}^{\infty}g_{3k}(x)h_{3k}(x)=\sum_{k=1}^{\infty}g_k(x)h_k(x).
 \end{align*}
 \end{proof}
 
 \subsection{Infinitesimal analysis}
 
 \begin{thm}
 \label{th:infinitesimal-rules}
 Let $\mathcal H[\varepsilon]$ denote the space of all elements of the form $x+\varepsilon\,y$ with $x,y\in\mathcal H$ and the infinitesimal $\varepsilon\neq 0$ satisfying the condition $\varepsilon^2=0$. Then for any smooth function $f$ it holds that 
\begin{align}
\label{eq:epsilon-order}
f(x+\varepsilon\,y)=f(x)+\varepsilon\,\langle \nabla f(x), y\rangle,\quad\text{for every}\;x,y\in\mathcal H.
\end{align} 
 In particular the following rules for differentiation hold: Let $f,g$ be smooth functions
 \begin{enumerate}
 \item $\nabla f(x)=0$, whenever $f$ is a constant function, and $\nabla cf(x)=c\nabla f(x)$ where $c\in\mathbb R$
 \item $\nabla (f+g)(x)=\nabla f(x)+\nabla g(x)$ (sum rule)
 \item $\nabla(f\cdot g)(x)=\nabla f(x)\cdot g(x)+ f(x)\cdot\nabla g(x)$ 
 (product rule). 
 \end{enumerate}
 
 \end{thm}
 
\begin{proof}
 By Theorem \ref{th:Taylor-expansion} there is a smooth mapping $g:\mathcal H\to\mathcal H\times\mathcal H$ such that 
 \begin{align*}
 f(x+y)=f(x)+\langle \nabla f(x), y\rangle + g(x)(y,y)=f(x)+\langle \nabla f(x), y\rangle + \sum_{i,j=1}^{\infty}y_iy_j\,g_{ij}(x).
 \end{align*}
Consequently 
\begin{align*}
f(x+\varepsilon\,y)=f(x)+\varepsilon\,\langle \nabla f(x), y\rangle + \varepsilon^2\,\sum_{i,j=1}^{\infty}y_iy_j\,g_{ij}(x)=f(x)+\varepsilon\,\langle\nabla f(x),y\rangle.
\end{align*}
The last equation can be written equivalently as follows
\begin{align*}
\langle\nabla f(x),y\rangle=\frac{f(x+\varepsilon\,y) - f(x)}{\varepsilon}.
\end{align*}
If $f$ is constant, then $\langle\nabla f(x),y\rangle=0$ for all $y\in\mathcal H$, and in particular for $y=\nabla f(x)$. This implies $\nabla f(x)=0$. Now let $c\in\mathbb R$ (say $c\neq 0$, else it is trivial), then 
\begin{align*}
\langle\nabla c\,f(x),y\rangle=\frac{c\,f(x+\varepsilon\,y) - c\,f(x)}{\varepsilon}=c\,\frac{f(x+\varepsilon\,y) - f(x)}{\varepsilon}=c\,\langle\nabla f(x),y\rangle=\langle c\,\nabla f(x),y\rangle
\end{align*}
for all $y\in\mathcal H$, in particular for $y=\nabla c\,f(x)-c\,\nabla f(x)$, implying $\|\nabla c\,f(x)-c\,\nabla f(x)\|^2=0$. Therefore $\nabla cf(x)=c\nabla f(x)$.
Now let $f,g$ be smooth, then
\begin{align*}
\langle\nabla (f+g)(x),y\rangle&=\frac{(f+g)(x+\varepsilon\,y) - (f+g)(x)}{\varepsilon}\\&
=\frac{f(x+\varepsilon\,y) + g(x+\varepsilon\,y) - f(x)-g(x)}{\varepsilon}\\&
=\frac{f(x+\varepsilon\,y) - f(x)}{\varepsilon}+\frac{g(x+\varepsilon\,y) - g(x)}{\varepsilon}\\&=\langle\nabla f(x),y\rangle+\langle\nabla g(x),y\rangle\\&=
\langle\nabla f(x) + \nabla g(x),y\rangle.
\end{align*}
Therefore $$\langle\nabla (f+g)(x)-\nabla f(x)-\nabla g(x),y\rangle=0$$ for all $y\in\mathcal H$, in particular for $y=\nabla (f+g)(x)-\nabla f(x)-\nabla g(x)$ implying $$\|\nabla (f+g)(x)-\nabla f(x)-\nabla g(x)\|=0.$$
Consequently $\nabla (f+g)(x)=\nabla f(x)+\nabla g(x)$.
Lastly we prove the product rule. Again let $f,g$ be real-valued smooth functions. Then 
\begin{align*}
\langle\nabla (f\cdot g)(x),y\rangle&=\frac{(f\cdot g)(x+\varepsilon\,y) - (f\cdot g)(x)}{\varepsilon}\\&
=\frac{f(x+\varepsilon\,y)\cdot g(x+\varepsilon\,y) - f(x)\cdot g(x)}{\varepsilon}\\&
=\frac{(f(x+\varepsilon\,y)-f(x))\cdot g(x+\varepsilon\,y) + f(x)\cdot(g(x+\varepsilon\,y) - g(x))}{\varepsilon}\\&
=\frac{f(x+\varepsilon\,y) - f(x)}{\varepsilon}\cdot g(x+\varepsilon\,y)+ f(x)\cdot \frac{g(x+\varepsilon\,y) - g(x)}{\varepsilon}\\&
=\langle\nabla f(x),y\rangle\cdot g(x+\varepsilon\,y)+ f(x)\cdot\langle\nabla g(x),y\rangle.
\end{align*}
From the above argument by Theorem \ref{th:Hadamard-lemma-2} we have that 
$g(x+\varepsilon\,y)=g(x)+\varepsilon\,\langle \nabla g(x),y\rangle$ for every $y\in\mathcal H$, so
\begin{align*}
\langle\nabla (f\cdot g)(x),y\rangle=\langle\nabla f(x),y\rangle\cdot g(x)+ f(x)\cdot\langle\nabla g(x),y\rangle+\varepsilon\,\langle\nabla f(x),y\rangle\langle\nabla g(x),y\rangle.
\end{align*}
Multiplying both sides by $\varepsilon$ and using $\varepsilon^2=0$ yields
\begin{align*}
\varepsilon\,(\langle\nabla (f\cdot g)(x),y\rangle-\langle\nabla f(x),y\rangle\cdot g(x)- f(x)\cdot\langle\nabla g(x),y\rangle)=0.
\end{align*}
By definition $\varepsilon\neq 0$, then $\langle\nabla (f\cdot g)(x)-\nabla f(x)\cdot g(x)- f(x)\cdot\nabla g(x),y\rangle=0$ for all $y\in\mathcal H$, and in particular for $y=\langle\nabla (f\cdot g)(x)-\nabla f(x)\cdot g(x)- f(x)\cdot\nabla g(x)$, implying 
$$\|\langle\nabla (f\cdot g)(x)-\nabla f(x)\cdot g(x)- f(x)\cdot\nabla g(x)\|=0.$$
Consequently $\langle\nabla (f\cdot g)(x)=\nabla f(x)\cdot g(x)+ f(x)\cdot\nabla g(x)$ for all $x\in\mathcal H$. 
\end{proof} 

\begin{remark} We may identify a certain mapping $\Psi(f):\mathcal H[\varepsilon]\to\mathbb R[\varepsilon]$ for all $f\in C^{\infty}(\mathcal H)$, that gives an interpretation for smooth functions in $\mathcal H$ taking arguments from $\mathcal H[\varepsilon]$. The mapping $\Psi$ acts as 
$$\Psi(f)(x+\varepsilon\,y)=f(x)+\varepsilon\,\langle \nabla f(x), y\rangle.$$
It is well-behaved with respect to the product of smooth functions. If $f,g\in C^{\infty}(\mathcal H)$, then
\begin{align*}
\Psi(f\cdot g)(x+\varepsilon\,y)=(f\cdot g)(x)+\varepsilon\,\langle \nabla (f\cdot g)(x), y\rangle.
\end{align*}
By rule (3) in Theorem \ref{th:infinitesimal-rules} we have 
\begin{align*}
\Psi(f\cdot g)(x+\varepsilon\,y)&=f(x)\cdot g(x)+\varepsilon\,\langle \nabla f(x)\cdot g(x)+ f(x)\cdot\nabla g(x), y\rangle\\&
=f(x)\cdot g(x)+\varepsilon\,(\langle \nabla f(x), y\rangle\cdot g(x)+f(x)\cdot\langle \nabla g(x), y\rangle)\\&
=(f(x)+\varepsilon\,\langle \nabla f(x), y\rangle)\cdot (g(x)+\varepsilon\,\langle \nabla g(x), y\rangle)\\&
=\Psi(f)(x+\varepsilon\,y)\cdot\Psi(g)(x+\varepsilon\,y).
\end{align*}
Moreover $\Psi(1)=1$, where $f(x)=1$ is the multiplicative identity element in $C^{\infty}(\mathcal H)$.
Such maps $\Psi$ play an important role in the theory of $C^{\infty}$-rings, e.g. see \cite[\S I,\S II]{Moerdijk--Reyes}.
\end{remark}

 \begin{remark}
 \label{r:sep-Banach}
 The extension of Hadamard's Lemma and its consequences presented here hold as well in separable Banach spaces admitting a Schauder basis. 
 \end{remark}
 \bibliographystyle{plain}
\bibliography{references}

\end{document}